\documentclass[10pt]{article}
\usepackage[utf8]{inputenc}
\usepackage[T1]{fontenc}             
\usepackage{lmodern}
\usepackage[french,english]{babel}
\usepackage{amsmath}
\usepackage{amsthm}
\usepackage{amssymb}
\usepackage{mathrsfs}
\usepackage{geometry}
\usepackage{graphicx}
\usepackage{comment}
\usepackage{url}
\usepackage{titlesec}
\usepackage[all]{xy}
\usepackage{enumitem}
\usepackage{amsfonts,amsxtra}
\usepackage{stmaryrd} 
\usepackage{tikz}

\makeatletter
\newcommand{\address}[1]{\gdef\@address{#1}}
\newcommand{\email}[1]{\gdef\@email{\url{#1}}}
\newcommand{\@endstuff}{\par\vspace{\baselineskip}\noindent\small
\begin{tabular}{@{}l}\scshape\@address\\\textit{E-mail address:} \@email\end{tabular}}
\AtEndDocument{\@endstuff}
\makeatother

\usepackage{lipsum}

\usepackage[colorlinks=true]{hyperref}

\newtheorem{df}{Definition}[section]
\newtheorem{rem}[df]{Remark}
\newtheorem{ex}[df]{Example}
\newtheorem{thm}[df]{Theorem}
\newtheorem{pp}[df]{Proposition} 
\newtheorem{lm}[df]{Lemma}

\newtheorem{cor}[df]{Corollary}

\newtheorem{thm intro}{Theorem}

\def\so{\mathfrak{so}}

\def\sl{\mathfrak{sl}}

\def\sp{\mathfrak{sp}}

\def\gg{\mathfrak{g}}
\def\hh{\mathfrak{h}}

\def\ss{\mathfrak{s}}

\def\gt{\tilde{\gg}}

\def\osp{\mathfrak{osp}}

\title{\Large{\textbf{Geometric properties of special orthogonal representations associated to exceptional Lie superalgebras}}}
\author{Philippe Meyer}
\address{Mathematical Institute, University of
Oxford, Oxford, Oxfordshire OX2 6GG, UK}
\email{philippe.meyer@maths.ox.ac.uk}
\date{\vspace{-2.5ex}}

\begin{document}
\maketitle

\begin{abstract}
From an octonion algebra $\mathbb{O}$ over a field $k$ of characteristic not two or three, we show that the fundamental representation ${\rm Im}(\mathbb{O})$ of the derivation algebra ${\rm Der}(\mathbb{O})$ and the spinor representation $\mathbb{O}$ of $\so({\rm Im}(\mathbb{O}))$ are special orthogonal representations. They have particular geometric properties coming from their similarities with binary cubics and we show that the covariants of these representations and their Mathews identities are related to the Fano plane and the affine space $(\mathbb{Z}_2)^3$. This also permits to give constructions of exceptional Lie superalgebras.
\end{abstract}

\section{Introduction}
\let\thefootnote\relax\footnote{\textit{Key words}: exceptional Lie superalgebra $\cdot$ covariant $\cdot$ spinor representation $\cdot$ octonions $\cdot$ Mathews identities.}
\let\thefootnote\relax\footnote{\textit{2020 Mathematics Subject Classification}: 17A75, 17B25, 17B60, 53C27.}

The space of binary cubics, a symplectic representation of the Lie algebra $\sl(2,k)$, has particular symplectic properties \cite{Eisenstein1844}, \cite{SlupinskiStanton2012}. It admits three covariants, among the Hessien and the discriminant, satisfying remarkable geometric identities \cite{Mathews11}. This representation is an example of a larger class of representations sharing these properties: the special $\epsilon$-orthogonal representations of colour Lie algebras \cite{StantonSlupinski15}, \cite{Meyer19Kostant}. The terminology special comes from their role in symplectic geometry \cite{CS09}. A special $\epsilon$-orthogonal representation $V$ of a colour Lie algebra $\gg$ can be extended to define a colour Lie algebra of the form
\begin{equation*}
    \gt=\gg\oplus \sl(2,k) \otimes V\otimes k^2.
\end{equation*}
In this way, special symplectic representations of Lie algebras give rise to Lie algebras and special orthogonal representations of Lie algebras give rise to Lie superalgebras.
\vspace{0.1cm}

In this paper, from an octonion algebra $\mathbb{O}$ over $k$, we show that
\begin{itemize}
    \item a one parameter family of $4$-dimensional representations of $\sl(2,k)\times \sl(2,k)$ ;
    \item the $7$-dimensional fundamental representation ${\rm Im}(\mathbb{O})$ of the Lie algebra ${\rm Der}(\mathbb{O})$ ;
    \item the $8$-dimensional spinor representation $\mathbb{O}$ of the Lie algebra $\so({\rm Im}(\mathbb{O}))$
\end{itemize}
are special orthogonal representations and give rise to exceptional Lie superalgebras of type $D(2,1;\alpha)$, $G_3$ and $F_4$ (in the Kac notation \cite{VKac1975}). This is similar to various constructions from Sudbery \cite{Sudbery83}, Kamiya and Okubo \cite{KamiyaOkubo03} and Elduque \cite{Elduque04}.
\vspace{0.1cm}

We explicitly compute the covariants of these representations. In particular, we give formulae of the moment maps of ${\rm Im}(\mathbb{O})$ and $\mathbb{O}$ and we show that the trilinear covariant of ${\rm Im}(\mathbb{O})$ is, up to a constant, the associator. The quadrilinear covariant of ${\rm Im}(\mathbb{O})$ admits a decomposition into a sum of $7$ decomposable forms which naturally correspond to the $7$ lines of the Fano plane and the two maps of the first Mathews identity are, up to constants, the Hodge duals of the cross-product on ${\rm Im}(\mathbb{O})$. Then we give a decomposition of the quadrilinear covariant of $\mathbb{O}$ into a sum of $14$ decomposable forms which naturally correspond to the $14$ affine planes of the affine space $(\mathbb{Z}_2)^3$. The two maps of the first Mathews identity are, up to a constant, the Hodge duals of the trilinear covariant of $\mathbb{O}$ and the two maps of the second Mathews identity are, up to constants, the Hodge duals of the moment map of $\mathbb{O}$.

For special orthogonal representations associated to basic classical Lie superalgebras and interpretation of their covariants, see the appendix of \cite{Meyer19Kostant}.

\subsection*{Acknowledgements} The author wants to express his gratitude to Marcus J. Slupinski for his suggestions and his encouragement. This work was partially supported by the Engineering and Physical Sciences Research Council grant [EP/N033922/1].

\subsection*{Notation} Let $k$ be a field of characteristic not two or three.
\vspace{0.2cm}

\noindent
For a finite-dimensional quadratic vector space $(V,q)$ and $i\in \mathbb{N}$ such that $i<char(k)$ if $0<char(k)$, we denote by $\eta : \Lambda^i(V)\rightarrow\Lambda^i(V)^*$ the canonical isomorphism given by the determinant and we consider the quadratic form $q_{\Lambda}$ (resp. $q_{\Lambda*}$) and the symmetric bilinear form $B_{\Lambda}$ associated by polarisation (resp. $B_{\Lambda*}$) on $\Lambda^i(V)$ (resp. $\Lambda^i(V)^*$) given by $\eta$.
\vspace{0.2cm}

\noindent
If $\lbrace e_i \rbrace$ is a basis of $V$, we denote $e_{i_1}\wedge\ldots\wedge e_{i_n}$ by $e_{i_1\ldots i_n}$.

\section{Lie superalgebras from special orthogonal representations}\label{section LSAs from special orthogonal rep}

In this section we explain how to construct a quadratic Lie superalgebra from an orthogonal representation of a quadratic Lie algebra, for details and proofs see \cite{Meyer19Kostant}. Let $(\gg,B_{\gg})$ be a finite-dimensional quadratic Lie algebra, let $(V,(\phantom{v},\phantom{v}))$ be a finite-dimensional quadratic vector space and let $\rho : \gg \rightarrow \so(V,(\phantom{v},\phantom{v}))$ be an orthogonal representation of $\gg$.
\vspace{0.1cm}

The moment map of the representation $\rho : \gg \rightarrow \so(V,(\phantom{v},\phantom{v}))$ is the $\gg$-equivariant alternating map $\mu\in {\rm Alt}_2(V,\gg)$ satisfying
$$B_{\gg}(x,\mu(v,w))=(\rho(x)(v),w) \qquad \forall x \in \gg, ~ \forall v,w \in V.$$
The standard example is the moment map of the fundamental representation of $\so(V,(\phantom{v},\phantom{v}))$:

\begin{ex}
Suppose that $\gg=\so(V,(\phantom{v},\phantom{v}))$ and $B_{\gg}(f,g)=-\frac{1}{2}Tr(fg)$ for all $f,g \in \so(V,(\phantom{v},\phantom{v}))$. The corresponding moment map $\mu_{can}\in{\rm Alt}_2(V,\so(V,(\phantom{v},\phantom{v})))$ satisfies
\begin{equation} \label{canonical moment map}
\mu_{can}(u,v)(w)=(u,w)v-(v,w)u \qquad \forall u,v,w\in V,
\end{equation}
and is a $\gg$-equivariant isomorphism between $\Lambda^2(V)$ and $\so(V,(\phantom{v},\phantom{v}))$. 
\end{ex}

We now define a particular class of orthogonal representations of quadratic Lie algebras:

\begin{df}
The representation $\rho : \gg \rightarrow \so(V,(\phantom{v},\phantom{v}))$ is said to be special orthogonal if
\begin{equation}\label{eq CS}
    \mu(u,v)(w)+\mu(u,w)(v)=(u,v)w+(u,w)v-2(v,w)u \qquad \forall u,v,w\in V.
\end{equation}
\end{df}

Special orthogonal representations can be extended to define Lie superalgebras as follows:

\begin{thm}\label{thm LSA from special orthogonal rep}
Let $\rho : \gg \rightarrow \so(V,(\phantom{v},\phantom{v}))$ be a finite-dimensional orthogonal representation of a finite-dimensional quadratic Lie algebra $(\gg,B_{\gg})$ and let $\sl(2,k)\rightarrow\sp(k^2,\omega)$ be the symplectic fundamental representation of the quadratic Lie algebra $(\sl(2,k),B_{\ss})$ where $\omega$ is the canonical symplectic form on $k^2$ and where $B_{\ss}(f,g)=\frac{1}{2}Tr(fg)$ for all $f,g\in \sl(2,k)$. Let  $\gt$ be the super vector space defined by
$$\gt:=\gg\oplus\sl(2,k)\oplus V\otimes k^2,$$
and let $B_{\gt}:=B_{\gg}\perp B_{\ss}\perp (\phantom{v},\phantom{v})\otimes \omega$. Then $(\gt,B_{\gt},\lbrace \phantom{v},\phantom{v} \rbrace)$ is a quadratic Lie superalgebra extending the bracket of $\gg\oplus\sl(2,k)$ and the action of $\gg\oplus\sl(2,k)$ on $V\otimes k^2$ if and only if $\rho : \gg \rightarrow \so(V,(\phantom{v},\phantom{v}))$ is a special orthogonal representation.
\end{thm}

In addition to the moment map, a trilinear and a quadrilinear alternating multilinear map can be naturally associated to a special orthogonal representation:

\begin{df} We define the multilinear alternating maps $\psi \in {\rm Alt}_3(V,V)$ and $Q\in {\rm Alt}_4(V,k)$ as follows:
\begin{align*}
\psi(v_1,v_2,v_3)&=\mu(v_1,v_2)(v_3)+\mu(v_3,v_1)(v_2)+\mu(v_2,v_3)(v_1), \\
Q(v_1,v_2,v_3,v_4)&=(v_1,\psi(v_2,v_3,v_4))-(v_4,\psi(v_1,v_2,v_3))+(v_3,\psi(v_4,v_1,v_2))-(v_2,\psi(v_3,v_4,v_1))
\end{align*}
for all $v_1,v_2,v_3,v_4 \in V$. The maps $\mu$, $\psi$ and $Q$ are called the covariants of $V$.
\end{df}

We have the following formulae:

\begin{pp}\label{pp formules lien BQ Bpsi dans le cas CS}
If $\rho : \gg \rightarrow \so(V,(\phantom{v},\phantom{v}))$ is special orthogonal, then we have
\begin{align*}
\psi(v_1,v_2,v_3)&=3(\mu(v_1,v_2)(v_3)-\mu_{can}(v_1,v_2)(v_3)), \\
Q(v_1,v_2,v_3,v_4)&=4(v_1,\psi(v_2,v_3,v_4)),
\end{align*}
for all $v_1,v_2,v_3,v_4 \in V$.
\end{pp}

For vector spaces $E,F,G,H$, the exterior product $f\wedge_{\phi} g\in{\rm Alt}_{p+q}(E,H)$ of $f\in {\rm Alt}_p(E,F)$ and $g\in {\rm Alt}_q(E,G)$ relative to a bilinear map $\phi : F\times G\rightarrow H$ is defined by
$$f\wedge_{\phi} g(v_1,\ldots,v_{p+q})=\sum_{\sigma\in S(\llbracket 1,p \rrbracket,\llbracket p+1,p+q\rrbracket)}sgn(\sigma)\phi(f(v_{\sigma(1)},\ldots,v_{\sigma(p)}),g(v_{\sigma(p+1)},\ldots,v_{\sigma(p+q)}))$$
where the sum is over the $(p,q)$-shuffle permutations in $S_{p+q}$. If $\phi$ is implicit, then we denote $f\wedge_{\phi} g$ by $f\wedge g$. The composition $f\circ g\in{\rm Alt}_{pq}(G,F)$ of $f\in {\rm Alt}_p(E,F)$ and $g\in {\rm Alt}_q(G,E)$ is defined by
$$f\circ g(v_1,\ldots,v_{pq})=\sum_{\sigma\in S(\llbracket 1,q \rrbracket,\ldots,\llbracket p(q-1)+1,pq\rrbracket)}sgn(\sigma)f(g(v_{\sigma(1)},\ldots,v_{\sigma(q)}),\ldots,g(v_{\sigma(p(q-1)+1)},\ldots,v_{\sigma(pq)}))$$
where the sum is over the $(q,\ldots,q)$-shuffle permutations in $S_{pq}$.
\vspace{0.3cm}

Covariants of special orthogonal representations satisfy to the following Mathews identities:

\begin{thm} \label{thm id mathews}
Let $\rho : \gg \rightarrow \so(V,(\phantom{v},\phantom{v}))$ be a finite-dimensional special orthogonal representation of a finite-dimensional quadratic Lie algebra and let $\mu\in {\rm Alt}_2(V,\gg)$, $\psi\in {\rm Alt}_3(V,V)$ and $Q\in {\rm Alt}_4(V,k)$ be its covariants. We have the following identities:
\begin{alignat}{4}
a)&\qquad\qquad\qquad\qquad&\mu\wedge_{\rho} \psi&=-\frac{3}{2}Q\wedge Id_V \quad &&\in {\rm Alt}_5(V,V), \label{id premat} \\
b)&\qquad\qquad\qquad\qquad&\mu \circ \psi &= 3Q \wedge \mu \quad &&\in {\rm Alt}_6(V,\gg), \label{id mat1} \\
c)&\qquad\qquad\qquad\qquad&\psi\circ \psi&=-\frac{27}{2} Q\wedge Q\wedge Id_V \quad &&\in {\rm Alt}_9(V,V), \label{id mat2} \\
d)&\qquad\qquad\qquad\qquad& Q\circ \psi&=-54 Q\wedge Q \wedge Q \quad &&\in {\rm Alt}_{12}(V,k). \label{id mat3}
\end{alignat}
\end{thm}

\section{A one-parameter family of special orthogonal representations of $\mathfrak{sl}(2,k)\times\mathfrak{sl}(2,k)$}\label{section one-parameter family of special orthogonal representations}

In this section we show that with respect to a one parameter family of invariant quadratic forms on $\sl(2,k)\times\sl(2,k)$, the tensor product of the two fundamental representations is a special orthogonal representation.
\vspace{0.2cm}

Let $(V,\omega_V)$ and $(W,\omega_W)$ be two-dimensional symplectic vector spaces. The vector space $V\otimes W$ is quadratic for the symmetric bilinear form $\omega_V\otimes \omega_W$ given by
$$\omega_V\otimes \omega_W(v_1\otimes w_1,v_2\otimes w_2)=-\omega_V(v_1,v_2)\omega_W(w_1,w_2) \qquad \forall v_1,v_2\in V, ~ \forall w_1,w_2\in W.$$
Consider the bilinear form $K_V$ (resp. $K_W$) on $\sp(V,\omega_V)$ (resp. $\sp(W,\omega_W)$) defined by $K_V(f,g)=\frac{1}{2}Tr(fg)$ (resp. $K_W(f,g)=\frac{1}{2}Tr(fg)$) for all $f,g \in \sp(V,\omega_V)$ (resp. $\sp(W,\omega_W)$). For $\alpha,\beta\in k^*$, we now consider the orthogonal representation
$$\sp(V,\omega_V)\times \sp(W,\omega_W)\rightarrow \so(V\otimes W, \omega_V\otimes \omega_W)$$
of the quadratic Lie algebra $(\sp(V,\omega_V)\times \sp(W,\omega_W),\frac{1}{\alpha}K_V\perp \frac{1}{\beta}K_W)$. Its moment map
$$\mu_{\alpha,\beta}:{\rm Alt}_2(V\otimes W,\sp(V,\omega_V)\times \sp(W,\omega_W))$$
satisfies
$$\mu_{\alpha,\beta}(v_1\otimes w_1,v_2\otimes w_2)=-\Big(\alpha \mu_V(v_1,v_2)\omega_W(w_1,w_2)+\beta \mu_W(w_1,w_2)\omega_V(v_1,v_2)\Big)\qquad \forall v_1,v_2\in V, ~ \forall w_1,w_2\in W,$$
where $\mu_i : S^2(V_i) \rightarrow \sp(V_i,\omega_i)$ is the canonical symmetric moment map given by
$$\mu_i(v_1,v_2)(v_3)=-\omega_i(v_1,v_3)v_2-\omega_i(v_2,v_3)v_1 \qquad \forall v_1,v_2,v_3 \in V_i.$$

\begin{pp}
The orthogonal representation
$$\sp(V,\omega_V)\times \sp(W,\omega_W)\rightarrow \so(V\otimes W, \omega_V\otimes \omega_W)$$
of the quadratic Lie algebra $(\sp(V,\omega_V)\times \sp(W,\omega_W),\frac{1}{\alpha}K_V\perp \frac{1}{\beta}K_W)$ is a special orthogonal representation if and only if $\alpha+\beta=-1$.
\end{pp}

\begin{proof}
Let $v_1\otimes w_1,v_2\otimes w_2, v_3\otimes w_3 \in V\otimes W$. We want to know under what conditions on $\alpha$ and $\beta$ do we have
\begin{align} \label{CS sl2xsl2}
&\mu_{\alpha,\beta}(v_1\otimes w_1,v_2\otimes w_2)(v_3\otimes w_3)+\mu_{\alpha,\beta}(v_1\otimes w_1,v_3\otimes w_3)(v_2\otimes w_2)=\omega_V\otimes \omega_W(v_1\otimes w_1,v_2\otimes w_2)v_3\otimes w_3 \nonumber \\
&+\omega_V\otimes \omega_W(v_1\otimes w_1,v_3\otimes w_3)v_2\otimes w_2-2\omega_V\otimes \omega_W(v_2\otimes w_2,v_3\otimes w_3)v_1\otimes w_1.
\end{align}
Since $V$ and $W$ are two-dimensional (and after a permutation of $v_1,v_2,v_3$ or $w_1,w_2,w_3$ if necessary) we have $v_3=av_1+bv_2$ and $w_3=cw_1+dw_2$ where $a,b,c,d \in k$. Hence we have
\begin{align*}
&\omega_V\otimes \omega_W(v_1\otimes w_1,v_2\otimes w_2)v_3\otimes w_3+\omega_V\otimes \omega_W(v_1\otimes w_1,v_3\otimes w_3)v_2\otimes w_2-2\omega_V\otimes \omega_W(v_2\otimes w_2,v_3\otimes w_3)v_1\otimes w_1 \\
&=\omega_V(v_1,v_2)\omega_W(w_1,w_2)\Big(-av_1\otimes dw_2-bv_2\otimes cw_1-2bv_2\otimes dw_2+av_1\otimes cw_1 \Big)
\end{align*}
On the other hand we have
\begin{align*}
&\mu_{\alpha,\beta}(v_1\otimes w_1,v_2\otimes w_2)(v_3\otimes w_3)+\mu_{\alpha,\beta}(v_1\otimes w_1,v_3\otimes w_3)(v_2\otimes w_2)\\
&=-\Big(\alpha\mu_V(v_1,v_2)(v_3)\otimes \omega_W(w_1,w_2)w_3+\beta\omega_V(v_1,v_2)v_3\otimes \mu_W(w_1,w_2)(w_3)+\alpha\mu_V(v_1,v_3)(v_2)\otimes \omega_W(w_1,w_3)w_2\\
&+\beta\omega_V(v_1,v_3)v_2\otimes \mu_W(w_1,w_3)(w_2)\Big)\\
&=(\alpha+\beta)\omega_V(v_1,v_3)v_2\otimes \omega_W(w_1,w_2)w_3+(\alpha+\beta)\omega_V(v_1,v_2)v_3\otimes \omega_W(w_1,w_3)w_2+\alpha\omega_V(v_2,v_3)v_1\otimes \omega_W(w_1,w_2)w_3\\
&+\alpha\omega_V(v_3,v_2)v_1\otimes \omega_W(w_1,w_3)w_2+\beta\omega_V(v_1,v_2)v_3\otimes \omega_W(w_2,w_3)w_1+\beta\omega_V(v_1,v_3)v_2\otimes \omega_W(w_3,w_2)w_1\\
&=(\alpha+\beta)\omega_V(v_1,v_2)bv_2\otimes \omega_W(w_1,w_2)cw_1+(\alpha+\beta)\omega_V(v_1,v_2)bv_2\otimes \omega_W(w_1,w_2)dw_2\\
&+(\alpha+\beta)\omega_V(v_1,v_2)av_1\otimes \omega_W(w_1,w_2)dw_2+(\alpha+\beta)\omega_V(v_1,v_2)bv_2\otimes \omega_W(w_1,w_2)dw_2\\
&+\alpha\omega_V(v_2,v_1)av_1\otimes \omega_W(w_1,w_2)cw_1+\alpha\omega_V(v_2,v_1)av_1\otimes \omega_W(w_1,w_2)dw_2+\alpha\omega_V(v_1,v_2)av_1\otimes \omega_W(w_1,w_2)dw_2\\
&+\beta\omega_V(v_1,v_2)av_1\otimes \omega_W(w_2,w_1)cw_1+\beta\omega_V(v_1,v_2)bv_2\otimes \omega_W(w_2,w_1)cw_1+\beta\omega_V(v_1,v_2)bv_2\otimes \omega_W(w_1,w_2)cw_1\\
&=(\alpha+\beta)\omega_V(v_1,v_2)\omega_W(w_1,w_2)\Big( bv_2\otimes cw_1+2bv_2\otimes dw_2+av_1\otimes dw_2-av_1\otimes cw_1\Big).
\end{align*}
Hence, Equation \eqref{CS sl2xsl2} is satisfied if and only if $\alpha+\beta=-1$ and so the representation $\sp(V,\omega_V)\times \sp(W,\omega_W)\rightarrow \so(V\otimes W, \omega_V\otimes \omega_W)$ is special orthogonal if and only if $\alpha+\beta=-1$.
\end{proof}
\vspace{0.2cm}

Suppose that $\alpha+\beta=-1$. By the previous proposition and Theorem \ref{thm LSA from special orthogonal rep} we have a Lie superalgebra $\gt_{\alpha}$ of the form
$$\gt_{\alpha}=\sp(V,\omega_V)\oplus \sp(W,\omega_W)\oplus \sl(2,k)\oplus V\otimes W \otimes k^2.$$
This a simple Lie superalgebra of type $D(2,1;\alpha)$ which is an exceptional simple Lie superalgebra if $\alpha$ is not equal to $-\frac{1}{2},-2$ or $1$.
\vspace{0.2cm}

\begin{rem}
\begin{enumerate}[label=\alph*)]
\item In \cite{Serganova1983}, Serganova shows that there are three families of simple real Lie superalgebras which are real forms of $D(2,1;\alpha)$ (see also \cite{Parker80} for a discussion about the real forms of $D(2,1;\alpha)$). If $k=\mathbb{R}$, the family $\gt_{\alpha}$ defined above corresponds to one these families.
\item There is a symmetry exchanging $\alpha$ and $\beta$. Hence, the special orthogonal representations $\sp(V,\omega_V)\times \sp(W,\omega_W)\rightarrow \so(V\otimes W, \omega_V\otimes \omega_W)$ of the quadratic Lie algebras $(\sp(V,\omega_V)\times \sp(W,\omega_W),\frac{1}{\alpha}K_V\perp \frac{1}{-1-\alpha}K_W)$ and $(\sp(V,\omega_V)\times \sp(W,\omega_W),\frac{1}{-1-\alpha}K_V\perp \frac{1}{\alpha}K_W)$ give rise to isomorphic Lie superalgebras $\gt_{\alpha}$ and $\gt_{-1-\alpha}$.
\item There is a singular case when $\alpha=\beta=-\frac{1}{2}$. The Lie algebra $\sp(V,\omega_V)\times\sp(W,\omega_W)$ is isomorphic to $\so(W_0,(\phantom{v},\phantom{v}))$, where $(W_0,(\phantom{v},\phantom{v}))$ is a four-dimensional hyperbolic vector space, and under this isomorphism, the quadratic form $\frac{1}{\alpha}K_V+\frac{1}{\alpha}K_W$ of $\sp(V,\omega_V)\times\sp(W,\omega_W)$ is isometric to the quadratic form $-\frac{1}{2}Tr(fg)$ for all $f,g \in \so(W_0,(\phantom{v},\phantom{v}))$.  Hence, we have that $\gt_{-\frac{1}{2}}$ is isomorphic to $\osp(W_0\oplus W_1,(\phantom{v},\phantom{v})\perp \omega)$ where $(W_1,\omega)$ is a two-dimensional symplectic vector space.
\end{enumerate}
\end{rem}

We now study the trilinear covariant and the quadrilinear covariant of the special orthogonal representation $\sp(V,\omega_V)\times \sp(W,\omega_W)\rightarrow \so(V\otimes W, \omega_V\otimes \omega_W)$. Note that the Mathews identities of Theorem \ref{thm id mathews} vanish identically because $V\otimes W$ is of dimension four.

\begin{pp}
Suppose that $\alpha+\beta=-1$. The trilinear covariant $\psi\in {\rm Alt}_3(V\otimes W,V\otimes W)$ and the quadrilinear covariant $Q\in {\rm Alt}_4(V\otimes W,k)$ of the special orthogonal representation
$$\sp(V,\omega_V)\times \sp(W,\omega_W)\rightarrow \so(V\otimes W, \omega_V\otimes \omega_W)$$
satisfies:
\begin{align*}
    \psi(v_1\otimes w_1,v_2\otimes w_2,v_3\otimes w_3)=&3(2\alpha+1)\Big( \omega_V(v_1,v_3)v_2\otimes \omega_W(w_3,w_2)w_1+\omega_V(v_2,v_3)v_1\otimes\omega_W(w_1,w_3)w_2\Big),\\
    Q(v_1\otimes w_1,v_2\otimes w_2,v_3\otimes w_3,v_4\otimes w_4)=&-12(2\alpha+1)\Big(\omega_V(v_2,v_4)\omega_V(v_1,v_3)\omega_W(w_4,w_3)\omega_W(w_1,w_2)\\
    &+\omega_V(v_3,v_4)\omega_V(v_1,v_2)\omega_W(w_2,w_4)\omega_W(w_1,w_3)\Big),
\end{align*}
for all $v_1,v_2,v_3,v_4\in V$, $w_1,w_2,w_3,w_4\in W$. 
\end{pp}

\begin{proof}
Let $v_1\otimes w_1,v_2\otimes w_2,v_3\otimes w_3 \in V\otimes W$. Since $V$ and $W$ are two-dimensional (and after a permutation of $v_1,v_2,v_3$ or $w_1,w_2,w_3$ if necessary) we have $v_3=av_1+bv_2$ and $w_3=cw_1+dw_2$ where $a,b,c,d \in k$. We have
\begin{align}
\mu_{\alpha,\beta}(v_1\otimes w_1,v_2\otimes w_2)(v_3\otimes w_3)&=(\alpha+\beta)\omega_V(v_2,v_3)v_1\otimes\omega_W(w_3,w_2)w_1+(\alpha+\beta)\omega_V(v_1,v_3)v_2\otimes\omega_W(w_1,w_3)w_2\nonumber\\
&+(\alpha-\beta)\omega_V(v_2,v_3)v_1\otimes\omega_W(w_1,w_3)w_2+(\alpha-\beta)\omega_V(v_1,v_3)v_2\otimes\omega_W(w_3,w_2)w_1,\label{Bpsi D21 rel 1}\\
\mu_{\alpha,\beta}(v_2\otimes w_2,v_3\otimes w_3)(v_1\otimes w_1)&=(\alpha+\beta)\omega_V(v_2,v_3)v_1\otimes\omega_W(w_2,w_3)w_1+2\alpha\omega_V(v_3,v_1)v_2\otimes\omega_W(w_2,w_3)w_1\nonumber\\
&+2\beta\omega_V(v_2,v_3)v_1\otimes\omega_W(w_3,w_1)w_2,\label{Bpsi D21 rel 2}\\
\mu_{\alpha,\beta}(v_3\otimes w_3,v_1\otimes w_1)(v_2\otimes w_2)&=(\alpha+\beta)\omega_V(v_1,v_3)v_2\otimes\omega_W(w_3,w_1)w_2+2\alpha\omega_V(v_3,v_2)v_1\otimes\omega_W(w_3,w_1)w_2\nonumber\\
&+2\beta\omega_V(v_3,v_1)v_2\otimes\omega_W(w_3,w_2)w_1.\label{Bpsi D21 rel 3}
\end{align}
Hence, summing Equations \eqref{Bpsi D21 rel 1}, \eqref{Bpsi D21 rel 2} and \eqref{Bpsi D21 rel 3}, we obtain
$$\psi(v_1\otimes w_1,v_2\otimes w_2,v_3\otimes w_3)=3(\alpha-\beta)\Big( \omega_V(v_1,v_3)v_2\otimes \omega_W(w_3,w_2)w_1+\omega_V(v_2,v_3)v_1\otimes\omega_W(w_1,w_3)w_2\Big).$$
The formula for $Q$ follows by Proposition \ref{pp formules lien BQ Bpsi dans le cas CS}
\end{proof}

\begin{rem}
For the singular case $\alpha=-\frac{1}{2}$, we have that the covariants $\psi$ and $Q$ vanish identically. It means that the representation $\sp(V,\omega_V)\times \sp(W,\omega_W)\rightarrow \so(V\otimes W, \omega_V\otimes \omega_W)$ is of $\mathbb{Z}_2$-Lie type in the sense of Kostant \cite{Kostant99} and then can be extended to define a Lie algebra structure on $\sp(V,\omega_V)\oplus \sp(W,\omega_W) \oplus V\otimes W$. This Lie algebra is isomorphic to the orthogonal Lie algebra $\so(V\otimes W \oplus L,\omega_V\otimes \omega_W\perp (\phantom{u},\phantom{v})_L)$ where $(L,(\phantom{u},\phantom{v})_L)$ is a one-dimensional quadratic vector space.
\end{rem}

\section{The fundamental representation of $G_2$ is special orthogonal}\label{section G3}


In this section, we show that the irreducible $7$-dimensional fundamental representation of an exceptional Lie algebra $\gg$ of type $G_2$ is special orthogonal. To do this we realise $\gg$ as the derivation algebra of an octonion algebra $\mathbb{O}$ and use octonionic calculations. We first recall some properties of the octonions, for details and proofs see \cite{Schafer66} and \cite{Springer00}.
\vspace{0.2cm}

Let $\mathbb{O}$ be an octonion (or Cayley) algebra over $k$. This is a 8-dimensional unital composition algebra, the conjugation $\bar{\phantom{v}}$ satisfies $q(u)=u\bar{u}$ for all $u\in \mathbb{O}$, where $q$ is the norm of $\mathbb{O}$, and we have $\mathbb{O}=k\oplus {\rm Im}(\mathbb{O})$, where ${\rm Im}(\mathbb{O})=\lbrace u \in \mathbb{O} ~ | ~ \bar{u}=-u \rbrace$. Denote $B$ the symmetric bilinear form associated by polarisation to $q$. Let $e_1,e_2,e_4 \in {\rm Im}(\mathbb{O})$ be such that $\mathcal{B}=\lbrace e_1,e_2,e_1e_2,e_4,e_1e_4,e_2e_4,(e_1e_2)e_4 \rbrace$ is an orthogonal and anisotropic basis of ${\rm Im}(\mathbb{O})$ and set $e_3:=e_1e_2$, $e_5:=e_1e_4$, $e_6:=e_2e_4$, $e_7:=(e_1e_2)e_4$. This basis is related to the Fano plane:
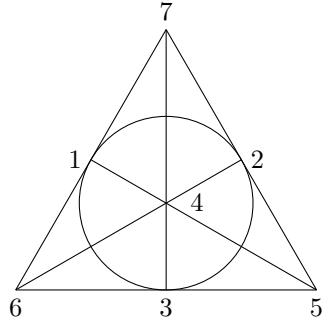
\begin{center}
\begin{tikzpicture}
\draw (0,0) -- (4,0);
\draw (0,0) -- (2,3.46410161514);
\draw (2,3.46410161514) -- (4,0);
\draw (2,1.15470053838) circle (1.15470053838) ;
\draw (2,3.46410161514) -- (2,0);
\draw (0,0) -- (3,1.73205080757);
\draw (4,0) -- (1,1.73205080757);
\draw (0,0) node[below]{$6$};
\draw (2,0) node[below]{$3$};
\draw (4,0) node[below]{$5$};
\draw (1,1.73205080757) node[left]{$1$};
\draw (3,1.73205080757) node[right]{$2$};
\draw (2,3.46410161514) node[above]{$7$};
\draw (2.2,1.15470053838) node[right]{$4$};
\end{tikzpicture}
\end{center}
in the sense that, for $i\neq j$, the product between $e_i$ and $e_j$ is a multiple of $e_k$ where $k$ is the third point on the line going through $i$ and $j$.
\vspace{0.2cm}

The commutator and the associator are the alternating maps given by:
\begin{align*}
[u,v]&=uv-vu,\\
(u,v,w)&=(uv)w-u(vw)
\end{align*}
for all $u,v,w\in \mathbb{O}$. The commutator doesn't define a Lie algebra structure on $\mathbb{O}$ since the Jacobi tensor $J$ satisfies
\begin{equation}\label{eq Jacobi G2}
J(u,v,w)=[u,[v,w]]+[v,[w,u]]+[w,[u,v]]=-6(u,v,w) \qquad \forall u,v,w\in \mathbb{O}.
\end{equation}
There is a cross-product on $\mathbb{O}$ defined by
$$u\times v=\frac{1}{2}(\bar{v}u-\bar{u}v)\qquad \forall u,v\in\mathbb{O},$$
and we have
\begin{alignat}{2}
q(u\times v)&=q_{\Lambda}(u\wedge v)=q(u)q(v)-B(u,v)^2 \qquad &&\forall u,v\in \mathbb{O},\label{eq length cross product}\\
u\times v&=\frac{1}{2}[u,v]=uv+B(u,v) \qquad &&\forall u,v,w\in {\rm Im}(\mathbb{O}),\label{eq link product and cross product}\\
u\times(v\times w)+v\times(u\times w)&=B(v,w)u+B(u,w)v-2B(u,v)w \qquad &&\forall u,v,w\in {\rm Im}(\mathbb{O}).\label{eq Malcev id}
\end{alignat}
The associative form $\phi$ on ${\rm Im}(\mathbb{O})$ is the trilinear alternating form defined by
$$\phi(u,v,w)=B(u,v\times w) \qquad \forall u,v,w\in {\rm Im}(\mathbb{O}),$$
and we have
\begin{align}
    \eta^{-1}(\phi)&=\frac{1}{q(e_1)q(e_2)}e_{123}-\frac{1}{q(e_1)q(e_2)q(e_4)}e_{167}+\frac{1}{q(e_1)q(e_2)q(e_4)}e_{257}-\frac{1}{q(e_1)q(e_2)q(e_4)}e_{356}\nonumber\\
    &+\frac{1}{q(e_1)q(e_4)}e_{145}+\frac{1}{q(e_2)q(e_4)}e_{246}+\frac{1}{q(e_1)q(e_2)q(e_4)}e_{347}.\label{decomp phi G2}
\end{align}

\vspace{0.2cm}

Let $\rho : {\rm Im}(\mathbb{O}) \rightarrow {\rm End}(\mathbb{O})$ be the map defined by $\rho(u)(x)=ux$ for $u \in {\rm Im}(\mathbb{O})$ and $x\in \mathbb{O}$. We have
$$\rho(u)^2=-q(u)Id \qquad \forall u\in {\rm Im}(\mathbb{O})$$
and so $\rho$ extends to the Clifford algebra $C({\rm Im}(\mathbb{O}),-q)$. The quantisation map $Q:\Lambda({\rm Im}(\mathbb{O}))\rightarrow C({\rm Im}(\mathbb{O}),-q)$ is an $O({\rm Im}(\mathbb{O}),B)$-equivariant isomorphism of vector spaces and then we have $C({\rm Im}(\mathbb{O}),-q)=\bigoplus\limits_i C^i({\rm Im}(\mathbb{O}),-q)$ where $C^i({\rm Im}(\mathbb{O}),-q)=Q(\Lambda^i({\rm Im}(\mathbb{O})))$. The map $\mu_{can}\circ Q^{-1} : C^2({\rm Im}(\mathbb{O}),-q) \rightarrow \so({\rm Im}(\mathbb{O}),q)$ is an isomorphism of Lie algebras. Let
$$\gg:=\lbrace x\in C^2({\rm Im}(\mathbb{O}),-q) ~ | ~ \rho(x)(1)=0 \rbrace.$$
This is a Lie algebra of type $G_2$, the map $\rho : \gg\rightarrow \so({\rm Im}(\mathbb{O}),q)$ is its $7$-dimensional fundamental representation and $\rho(\gg)$ is equal to the set of derivations of $\mathbb{O}$. Define the ad-invariant quadratic form $B_{\gg}$ on $\gg$ by
$$B_{\gg}(x,y)=-\frac{1}{3} Tr(\rho(x)\rho(y))\qquad\forall x,y\in \gg.$$

\begin{pp}\label{pp mu G2}
The moment map $\mu_{{\rm Im}}: \Lambda^2({\rm Im}(\mathbb{O}))\rightarrow \gg$ satisfies
$$\mu_{{\rm Im}}(u,v)(w)=-\frac{1}{4}([w,[u,v]]+3(u,v,w))\qquad\forall u,v,w\in {\rm Im}(\mathbb{O})$$
\end{pp}

\begin{proof}
For $u,v\in {\rm Im}(\mathbb{O})$, let $D(u,v)\in \gg$ be such that $\rho(D(u,v))(x)=[x,[u,v]]+3(u,v,x)$ for all $x\in {\rm Im}(\mathbb{O})$. Let $D$ in $\gg$. We want to show that
$$Tr(\rho(D)\rho(D(u,v))=12B(D(u),v).$$
Without loss of generality (changing $\mathcal{B}$ if necessary) we can assume that $u=e_1$ and $v=e_2$.
First of all
$$Tr(\rho(D)\rho(D(u,v))=\sum\limits_{e_i \in \mathcal{B}}\frac{1}{q(e_i)}B(\rho(D)(\rho(D(u,v))(e_i)),e_i).$$
We have
\begin{align*}
    \frac{1}{q(e_1)}B(D(D(e_1,e_2)(e_1)),e_1)&=4B(D(e_1),e_2),\\
    \frac{1}{q(e_1)}B(D(D(e_1,e_2)(e_2)),e_2)&=4B(D(e_1),e_2),\\
    \frac{1}{q(e_1e_2)}B(D(D(e_1,e_2)(e_1e_2)),e_1e_2)&=0,\\
    \frac{1}{q(e_3)}B(D(D(e_1,e_2)(e_3)),e_3)&=-\frac{2}{q(e_3)}B(D(e_3),(e_1e_2)e_3),\\
    \frac{1}{q(e_1e_3)}B(D(D(e_1,e_2)(e_1e_3)),e_1e_3)&=2B(D(e_1),e_2)+\frac{2}{q(e_3)}B(D(e_3),(e_1e_2)e_3),\\
    \frac{1}{q(e_2e_3)}B(D(D(e_1,e_2)(e_2e_3)),e_2e_3)&=2B(D(e_1),e_2)+\frac{2}{q(e_3)}B(D(e_3),(e_1e_2)e_3),\\
    \frac{1}{q((e_1e_2)e_3)}B(D(D(e_1,e_2)((e_1e_2)e_3)),(e_1e_2)e_3)&=-\frac{2}{q(e_3)}B(D(e_3),(e_1e_2)e_3),
\end{align*}

and hence
$$Tr(\rho(D)\rho(D(u,v)))=12B(D(u),v).$$
\end{proof}

\begin{cor} \label{expression mu g2 avec mucan}
For $u,v,w\in {\rm Im}(\mathbb{O})$, we have
\begin{enumerate}[label=\alph*)]
\item $\mu_{\rm Im}(u,v\times w)+\mu_{\rm Im}(w,u\times v)+\mu_{\rm Im}(v,w\times u)=0,$
\item $\mu_{{\rm Im}}(u,v)(w)=\frac{3}{2}\mu_{can}(u,v)(w)+\frac{1}{8}[w,[u,v]].$
\end{enumerate}
\end{cor}
\vspace{0.1cm}

\begin{proof}
$a)$ See (3.73) p.78 of \cite{Schafer66}.
\vspace{0.2cm}

\noindent
$b)$ Let $u,v,w\in {\rm Im}(\mathbb{O})$. We first show that
\begin{equation}\label{proof expression mu g2 avec mucan}
    -\frac{1}{4}[w,[u,v]]-\frac{1}{2}(u,v,w)=\mu_{can}(u,v)(w).
\end{equation}
Suppose that $u$ and $v$ are anisotropic and orthogonal.
\vspace{0.2cm}

\noindent
$\bullet$ If $w=u$ then \eqref{proof expression mu g2 avec mucan} follows from
$$-\frac{1}{4}[w,[u,v]]-\frac{1}{2}(u,v,w)=-\frac{1}{2}[u,uv]=-u^2v=q(u)v=\mu_{can}(u,v)(u).$$

\noindent
$\bullet$ If $w=uv$ then \eqref{proof expression mu g2 avec mucan} is clear since $[w,[u,v]]=(u,v,w)=\mu_{can}(u,v)(w)=0$.
\vspace{0.2cm}

\noindent
$\bullet$ If $\lbrace u,v,uv,w \rbrace$ are orthogonal then we have $\mu_{can}(u,v)(w)=0$ and
$$-\frac{1}{4}[w,[u,v]]=-w(uv)=(uv)w=\frac{1}{2}(u,v,w).$$
Hence \eqref{proof expression mu g2 avec mucan} is satisfied and this proves the corollary using Proposition \ref{pp mu G2}.
\end{proof}
\vspace{0.1cm}

We now give the main result of this section.
\vspace{0.1cm}

\begin{thm} \label{pp ImO is special}
The representation $\rho : \gg \rightarrow \so({\rm Im}(\mathbb{O}),B)$ of the quadratic Lie algebra $(\gg,B_{\gg})$ is a special orthogonal representation.
\end{thm}
\begin{proof}
Let $u,v,w\in {\rm Im}(\mathbb{O})$. Using Proposition \ref{pp mu G2} and \eqref{eq Malcev id} we have
\begin{align*}
    \mu_{\rm Im}(u,v)(w)+\mu_{\rm Im}(u,w)(v)&=-\frac{1}{4}([w,[u,v]]+[v,[u,w]])=-w\times (u\times v)-v\times(u\times w)\\
    &=w\times (v\times u)+v\times (w\times u)=B(u,v)w+B(u,w)v-2B(w,v)u.
\end{align*}
\end{proof}

By Theorems \ref{thm LSA from special orthogonal rep} and \ref{pp ImO is special} we have a Lie superalgebra $\gt$ of the form
$$\gt=\gg\oplus \sl(2,k)\oplus {\rm Im}(\mathbb{O}) \otimes k^2.$$
This is an exceptional simple Lie superalgebra of type $G_3$.

\begin{rem}
If $k=\mathbb{R}$, Serganova (see \cite{Serganova1983}) showed that there are two real forms of $G_3$ whose even parts are isomorphic to the compact (resp. split) exceptional simple real Lie algebra of type $G_2$ in direct sum with $\sl(2,\mathbb{R})$ and whose odd parts are isomorphic to the tensor product of the fundamental representations. In our construction, if $\mathbb{O}$ is the compact (resp. split) octonion algebra, the Lie algebra $\gg$ is the compact (resp. split) exceptional simple real Lie algebra of type $G_2$ and both real forms of $G_3$ are obtained by our construction.
\end{rem}

Since the representation $\gg \rightarrow \so({\rm Im}(\mathbb{O}),B)$ is special, we calculate its covariants and the Mathews identities they satisfy. Both-sides of Equations \eqref{id mat2} and \eqref{id mat3} vanish identically since ${\rm Im}(\mathbb{O})$ is of dimension $7$. It turns out, that both-sides of Equation \eqref{id mat1} also vanish identically. However, both sides of Equation \eqref{id premat} do not vanish identically and, up to constants, $Q_{\rm Im}\wedge Id$ and $\mu_{\rm Im}\wedge_{\rho} \psi_{\rm Im}\in {\rm Alt}_5({\rm Im}(\mathbb{O}),{\rm Im}(\mathbb{O}))$ are the Hodge duals of the cross-product $\times \in {\rm Alt}_2({\rm Im}(\mathbb{O}),{\rm Im}(\mathbb{O}))$.

\begin{pp}\label{pp covariants ImO}
Let $\mu_{\rm Im},\psi_{\rm Im},Q_{\rm Im}$ be the covariants of the special orthogonal representation $\rho : \gg \rightarrow \so({\rm Im}(\mathbb{O}),B)$. We have
\begin{enumerate}[label=\alph*)]
\item $\psi_{\rm Im}(v_1,v_2,v_3)=-\frac{3}{4}(v_1,v_2,v_3)~$ for all $v_1,v_2,v_3\in {\rm Im}(\mathbb{O})$,
\item $Q_{\rm Im}(v_1,v_2,v_3,v_4)=-3B(v_1,(v_2,v_3,v_4))~$ for all $v_1,v_2,v_3,v_4\in {\rm Im}(\mathbb{O})$,
\item  \begin{align}
        \eta^{-1}(Q_{\rm Im})&=\frac{6}{q(e_1)q(e_2)q(e_4)}e_{1247}-\frac{6}{q(e_1)q(e_2)q(e_4)}e_{1256}-\frac{6}{q(e_1)q(e_2)q(e_4)}e_{1346}-\frac{6}{q(e_1)^2q(e_2)q(e_4)}e_{1357}\nonumber\\
        &+\frac{6}{q(e_1)q(e_2)q(e_4)}e_{2345}-\frac{6}{q(e_1)q(e_2)^2q(e_4)}e_{2367}-\frac{6}{q(e_1)q(e_2)q(e_4)^2}e_{4567},\label{decomp Q G2}
        \end{align}
\item $\mu_{\rm Im} \circ \psi_{\rm Im}=0$ and $Q_{\rm Im}\wedge \mu_{\rm Im}=0$.
\end{enumerate}
\end{pp}
\begin{proof}
$a)$ By Proposition \ref{pp mu G2} and by Equation \eqref{eq Jacobi G2} we obtain
\begin{align*}
    \psi_{\rm Im}(v_1,v_2,v_3)&=-\frac{1}{4}(J(v_1,v_2,v_3)+3(v_1,v_2,v_3)+3(v_2,v_3,v_1)+3(v_3,v_1,v_2))\\
    &=-\frac{1}{4}(J(v_1,v_2,v_3)+9(v_1,v_2,v_3))\\
    &=-\frac{3}{4}(v_1,v_2,v_3).
\end{align*}

\noindent
$b)$ Follows from Proposition \ref{pp formules lien BQ Bpsi dans le cas CS}.
\vspace{0.2cm}

\noindent
$c)$ The decomposition follows from $b)$ and the fact that for $i_1<i_2<i_3<i_4$, then $Q_{\rm Im}(e_{i_1},e_{i_2},e_{i_3},e_{i_4})$ is non-zero if and only if $(i_1,i_2,i_3,i_4)\in \lbrace (1,2,4,7),(1,2,5,6),(1,3,4,6),(1,3,5,7),(2,3,4,5),(2,3,6,7),(4,5,6,7)\rbrace$.
\vspace{0.2cm}

\noindent
$d)$ Let $v_1,\ldots,v_6 \in {\rm Im}(\mathbb{O})$. We have
\begin{align*}
\mu_{\rm Im} \circ \psi_{\rm Im} (v_1,\ldots,v_6)&=\sum \limits_{\sigma \in S(\llbracket1,3\rrbracket,\llbracket 4,6 \rrbracket)}sgn(\sigma) \mu_{\rm Im} (\psi_{\rm Im}(v_{\sigma(1)},v_{\sigma(2)},v_{\sigma(3)}),\psi_{\rm Im}(v_{\sigma(4)},v_{\sigma(5)},v_{\sigma(6)})) \\
&=2\sum \limits_{\sigma \in S'}sgn(\sigma) \mu_{\rm Im} (\psi_{\rm Im}(v_{\sigma(1)},v_{\sigma(2)},v_{\sigma(3)}),\psi_{\rm Im}(v_{\sigma(4)},v_{\sigma(5)},v_{\sigma(6)}))\\
&=\frac{9}{8}\sum \limits_{\sigma \in S'}sgn(\sigma) \mu_{\rm Im} ((v_{\sigma(1)},v_{\sigma(2)},v_{\sigma(3)}),(v_{\sigma(4)},v_{\sigma(5)},v_{\sigma(6)}))
\end{align*}
where $S':=\lbrace Id,(14),(15),(16),(24),(25),(26),(34),(35),(36) \rbrace$.
\vspace{0.2cm}

Suppose that $v_i\in \mathcal{B}$ for all $i\in \llbracket 1,6 \rrbracket$. Since there is no distinguished way to choose 5 different points on the Fano plane, then, without loss of generality, we can assume that $v_i=e_i$ for all $i\in \llbracket 1,6 \rrbracket$. Since
$$(e_1,e_2,e_3)=(e_1,e_4,e_5)=(e_2,e_4,e_6)=(e_3,e_5,e_6)=0$$
then we have
\begin{align*}
    \mu_{\rm Im} \circ \psi_{\rm Im} (v_1,\ldots,v_6)&=-\frac{9}{8}\sum \limits_{\sigma \in S''}\mu_{\rm Im} ((e_{\sigma(1)},e_{\sigma(2)},e_{\sigma(3)}),(e_{\sigma(4)},e_{\sigma(5)},e_{\sigma(6)}))\\
    &=-\frac{9}{2}\sum \limits_{\sigma \in S''}\mu_{\rm Im} ((e_{\sigma(1)}e_{\sigma(2)})e_{\sigma(3)},(e_{\sigma(4)}e_{\sigma(5)})e_{\sigma(6)})
\end{align*}
where $S'':=\lbrace (14),(15),(24),(26),(35),(36) \rbrace$. Hence, we have that
$$\mu_{\rm Im} \circ \psi_{\rm Im}(v_1,\ldots,v_6)=-9q(e_1)q(e_2)q(e_4)\Big(\mu_{\rm Im} (e_1 e_2,e_4)+\mu_{\rm Im} (e_2 e_4,e_1)+\mu_{\rm Im} (e_4 e_1,e_2)\Big)$$
and so, using $a)$ of Corollary \ref{expression mu g2 avec mucan}, we obtain $\mu_{\rm Im} \circ \psi_{\rm Im}=0$ and by Theorem \ref{thm id mathews} we have $Q_{\rm Im}\wedge \mu_{\rm Im} =0$.
\end{proof}
\vspace{0.1cm}

\begin{rem}
\begin{enumerate}[label=\alph*)]
    \item We have
    $$\phi\wedge Q_{\rm Im}(e_1\wedge\ldots\wedge e_7)=-42q(e_1)^2q(e_2)^2q(e_4)^2.$$
    If $char(k)\neq 7$, then $\phi\wedge Q_{\rm Im}$ defines an orientation on ${\rm Im}(\mathbb{O})$
    \item In the decomposition \eqref{decomp phi G2} (resp. \eqref{decomp Q G2}), the seven quadruples of indices $\lbrace i_1,i_2,i_3,i_4\rbrace$ appearing are exactly (resp. the complements of) the seven lines of the Fano plane.
\end{enumerate}
\end{rem}

Suppose that $char(k)=0$ or $char(k)>7$. Define a quadratic form $B_{{\rm Alt}}$ on ${\rm Alt}_i({\rm Im}(\mathbb{O}),{\rm Im}(\mathbb{O}))\cong\Lambda^i({\rm Im}(\mathbb{O}))^*\otimes {\rm Im}(\mathbb{O})$ to be the tensor product of $B_{\Lambda^*}$ and $B$. For $f\in {\rm Alt}_i({\rm Im}(\mathbb{O}),{\rm Im}(\mathbb{O}))$ define its Hodge dual $*f\in {\rm Alt}_{7-i}({\rm Im}(\mathbb{O}),{\rm Im}(\mathbb{O}))$ to be the unique element which satisfies
$$\alpha\wedge_{B} *f=B_{{\rm Alt}}(\alpha,f)\phi\wedge Q_{\rm Im} \qquad \forall \alpha \in {\rm Alt}_i({\rm Im}(\mathbb{O}),{\rm Im}(\mathbb{O})).$$

\begin{pp}
We have
$$*\times=\frac{147}{8}Q_{\rm Im}\wedge Id=-\frac{49}{4}\mu_{\rm Im}\wedge_{\rho} \psi_{\rm Im}.$$
\end{pp}

\begin{proof}
Let $\alpha\in {\rm Alt}_2({\rm Im}(\mathbb{O}),{\rm Im}(\mathbb{O}))$. We have
\begin{equation}\label{eq proof hodge cross product Im O 1}
B_{{\rm Alt}}(\alpha,\times)\phi\wedge Q_{\rm Im}(e_1\wedge \ldots \wedge e_7)=-42\sum\limits_{i<j}\frac{1}{q(e_i)q(e_j)}B(\alpha(e_i,e_j),e_i\times e_j)q(e_1)^2q(e_2)^2q(e_4)^2.
\end{equation}
On the other hand
\begin{align*}
    \alpha\wedge_B Q_{\rm Im}\wedge Id(e_1\wedge\ldots\wedge e_7)&=\frac{1}{126}\sum\limits_{\sigma\in S_7}sgn(\sigma)B(\alpha(e_{\sigma(1)},e_{\sigma(2)}),e_{\sigma(3)})Q_{\rm Im}(e_{\sigma(4)},e_{\sigma(5)},e_{\sigma(6)},e_{\sigma(7)}).\\
\end{align*}
Since $\alpha$ and $Q_{\rm Im}$ are alternating and using the decomposition of Equation \eqref{decomp Q G2}, we have
\begin{equation}\label{eq proof hodge cross product Im O 2}
    \alpha\wedge_B Q_{\rm Im}\wedge Id(e_1\wedge\ldots\wedge e_7)=\frac{8}{21}\sum\limits_{\sigma\in S}sgn(\sigma)B(\alpha(e_{\sigma(1)},e_{\sigma(2)}),e_{\sigma(3)})Q_{\rm Im}(e_{\sigma(4)},e_{\sigma(5)},e_{\sigma(6)},e_{\sigma(7)})
\end{equation}
where
$$S=\lbrace\sigma\in S_7 ~ | ~ \sigma(1)<\sigma(2), \quad (\sigma(4),\sigma(5),\sigma(6),\sigma(7))\in \lbrace (1,2,4,7),\ldots,(4,5,6,7) \rbrace\rbrace.$$
We have $|S|=21$, each summand in \eqref{eq proof hodge cross product Im O 2} correspond to one summand in \eqref{eq proof hodge cross product Im O 1} and so, a straightforward calculation gives
$$\alpha\wedge_B Q_{\rm Im}\wedge Id(e_1\wedge\ldots\wedge e_7)=-\frac{16}{7}\sum\limits_{i<j}\frac{1}{q(e_i)q(e_j)}B(\alpha(e_i,e_j),e_i\times e_j)q(e_1)^2q(e_2)^2q(e_4)^2$$
and hence
$$\alpha\wedge_B Q_{\rm Im}\wedge Id=\frac{8}{147}B_{\rm Alt}(\alpha,\times)\phi \wedge Q_{\rm Im}.$$
\end{proof}

\begin{rem}
One can show similarly that, up to constants,
the identity is the Hodge dual of $\phi\wedge \psi_{\rm Im}$, the covariant $\mu_{\rm Im}$ is the Hodge dual of $\phi\wedge \mu_{\rm Im}$ and the covariant $\psi_{\rm Im}$ is the Hodge dual of $\phi\wedge Id$.
\end{rem}

\section{The spinor representation of a Lie algebra of type $\mathfrak{so}(7)$ is special orthogonal}\label{section F4}

In this section, we show that the $8$-dimensional spinor representation $\mathbb{O}$ of $C^2({\rm Im}(\mathbb{O}),-q)$ is special orthogonal. Let $e_2,e_3,e_5 \in {\rm Im}(\mathbb{O})$ be such that $\mathcal{B}=\lbrace 1,e_2,e_3,e_2e_3,e_5,e_2e_5,e_3e_5,(e_2e_3)e_5 \rbrace$ is an orthogonal and anisotropic basis of $\mathbb{O}$ and set $e_4:=e_2e_3$, $e_6:=e_2e_5$, $e_7:=e_3e_5$, $e_8:=(e_2e_3)e_5$.
\vspace{0.2cm}

Since the Clifford algebra $C({\rm Im}(\mathbb{O}),-q)$ is $\mathbb{Z}_2$-graded, it is a Lie superalgebra for the bracket given by
$$\lbrace c,d \rbrace:=cd-(-1)^{|c||d|}dc \qquad \forall c,d \in C$$
where $|c|$ and $|d|$ denotes the parity of homogeneous elements $c$ and $d$. Let $\hh:=C^2({\rm Im}(\mathbb{O}),-q)$ and define the ad-invariant quadratic form $B_{\hh}$ on $\hh$ by
$$B_{\hh}(x,y)=-\frac{3}{8} Tr(\rho(x)\rho(y)) \qquad \forall x,y\in \hh.$$

\begin{df}
Let $\Omega:=Q(\eta^{-1}(\phi))\in C({\rm Im}(\mathbb{O}),-q)$. For $u\in {\rm Im}(\mathbb{O})$, define $c_u \in \hh$ by $c_u:=\lbrace u,\Omega \rbrace$ and define the $7$-dimensional subspace $W \subset \hh$ by $W:={\rm span}<\lbrace c_u ~ | ~ u \in {\rm Im}(\mathbb{O})\rbrace>$.
\end{df}

The subspace $W$ acts on $\mathbb{O}$ as follows.

\begin{pp} \label{action c_u}
Let $u,v \in {\rm Im}(\mathbb{O})$. We have
$$\rho(c_u)(1)=-6u, \quad \rho(c_u)(v)=2u\times v+6B(u,v).$$
\end{pp}

\begin{proof}
Using \eqref{decomp phi G2}, we obtain
$$\rho(\Omega)(1)=-7, \qquad \rho(\Omega)(u)=u,$$
and so
$$\rho(c_u)(1)=\rho(u)(\rho(\Omega)(1))+\rho(\Omega)(\rho(u)(1))=-7u+u=-6u.$$
Using \eqref{eq link product and cross product}, we have
\begin{align*}
    \rho(c_u)(v)&=\rho(u)(\rho(\Omega)(v))+\rho(\Omega)(\rho(u)(v))\\
    &=uv+\rho(\Omega)(uv)\\
    &=u\times v-B(u,v)+\rho(\Omega)(u\times v)-B(u,v)\Omega(1)\\
    &=2u\times v+6B(u,v).
\end{align*}
\end{proof}

Now, we can express the moment map of $\mathbb{O}$ in terms of the moment map of ${\rm Im}(\mathbb{O})$ and $W$:
\vspace{0.1cm}

\begin{pp}\label{pp mu so(7)}
The moment map $\mu_{\mathbb{O}}: \Lambda^2(\mathbb{O})\rightarrow \hh$ satisfies to
\begin{alignat}{2}
\mu_{\mathbb{O}}(u,v)&=\frac{8}{9}\mu_{{\rm Im}}(u,v)+\frac{1}{18}c_{u\times v} \qquad &&\forall u,v\in {\rm Im}(\mathbb{O}),\\
\mu_{\mathbb{O}}(u,1)&=\frac{1}{6}c_u \qquad &&\forall u\in {\rm Im}(\mathbb{O}).
\end{alignat}
\end{pp}

\begin{proof} We first need the following lemma:
\begin{lm}
Let $u,v\in {\rm Im}(\mathbb{O})$ and $D\in \gg$. We have
$$Tr(\rho(c_u)\rho(c_v))=-96B(u,v), \qquad Tr(\rho(D)\rho(c_u))=0.$$
\end{lm}
\begin{proof}
Let $w\in {\rm Im}(\mathbb{O})$. We have
\begin{align*}
        B(\rho(c_u)(\rho(c_v)(w)),w)&=2B(\rho(c_u)(v\times w),w)+6B(v,w)B(\rho(c_u)(1),w)\\
        &=2B(2u\times(v\times w)+6B(u,v\times w),w)-36B(u,w)B(v,w)\\
        &=-4B(u\times w,v\times w)-36B(u,w)B(v,w).
\end{align*}
The linearisation of \eqref{eq length cross product} gives
$$B(u\times w,v\times w)=B(u,v)q(w)-B(u,w)B(v,w)$$
and so
$$B(\rho(c_u)(\rho(c_v)(w)),w)=-4B(u,v)q(w)-32B(u,w)B(v,w).$$
We also have
$$B(\rho(c_u)(\rho(c_v)(1)),1)=-6B(\rho(c_u)(v),1)=-36B(u,v).$$
If $u$ and $v$ are orthogonal, without loss of generality (changing $\mathcal{B}$ if necessary), we can assume that $u=e_2$ and $v=e_3$. Hence
$$Tr(\rho(c_u)\rho(c_v))=\sum\limits_{e_i \in \mathcal{B}}\frac{1}{q(e_i)}B(\rho(c_u)(\rho(c_v)(e_i)),e_i)=0,$$
similarly we have
$$Tr(\rho(c_u)^2)=-96q(u)$$
and so
$$Tr(\rho(c_u)\rho(c_v))=-96B(u,v).$$
A straightforward calculation shows that $\gg$ and $W$ are orthogonal.
\end{proof}
Let $D\in \gg$.
We have
\begin{align*}
    B_{\hh}(D,\mu_{\mathbb{O}}(u,v))=B(D(u),v)=B_{\gg}(D,\mu_{{\rm Im}}(u,v))=\frac{8}{9}B_{\hh}(D,\mu_{{\rm Im}}(u,v)),
\end{align*}
and, using the previous lemma, we also have
$$B_{\hh}(c_w,\mu_{\mathbb{O}}(u,v))=B(c_w(u),v)=2B(w,u\times v)=-\frac{1}{48}Tr(c_wc_{u\times v})=\frac{1}{18}B_{\hh}(c_w,c_{u\times v}),$$
and so
$$\mu_{\mathbb{O}}(u,v)=\frac{8}{9}\mu_{{\rm Im}}(u,v)+\frac{1}{18}c_{u\times v}.$$
Since $\rho(D)(1)=0$, then we have $B_{\hh}(D,\mu_{\mathbb{O}}(u,1))=0$. Moreover,
$$B_{\hh}(c_w,\mu(u,1))=B(c_w(u),1)=6B(u,w)=-\frac{1}{16}Tr(c_wc_u)=\frac{1}{6}B_{\hh}(c_w,c_u)$$
and so
$$\mu_{\mathbb{O}}(u,1)=\frac{1}{6}c_u.$$
\end{proof}

The counterpart of $a)$ of Corollary \ref{expression mu g2 avec mucan} is the following property about the moment map of $\hh$:

\begin{cor} We have
\begin{equation}\label{eq moment map O of cyclic sum}
    \mu_{\mathbb{O}}(u,v\times w)+\mu_{\mathbb{O}}(v,w\times u)+\mu_{\mathbb{O}}(w,u\times v)=-\frac{1}{2}\mu_{\mathbb{O}}((u,v,w),1)\qquad \forall u,v,w\in {\rm Im}(\mathbb{O}).
\end{equation}
\end{cor}

\begin{proof}
Using $a)$ of Corollary \ref{expression mu g2 avec mucan}, we have
\begin{equation*}
    \mu_{\mathbb{O}}(u,v\times w)+\mu_{\mathbb{O}}(v,w\times u)+\mu_{\mathbb{O}}(w,u\times v)=\frac{1}{18}c_{u\times (v\times w)+v\times(w\times u)+w\times(u\times v)}.
\end{equation*}
Since, using \eqref{eq Jacobi G2}, we have
$$u\times (v\times w)+v\times(w\times u)+w\times(u\times v)=\frac{1}{4}J(u,v,w)=-\frac{3}{2}(u,v,w)$$
then we obtain
\begin{equation*}
    \mu_{\mathbb{O}}(u,v\times w)+\mu_{\mathbb{O}}(v,w\times u)+\mu_{\mathbb{O}}(w,u\times v)=-\frac{1}{12}c_{(u,v,w)}=-\frac{1}{2}\mu_{\mathbb{O}}((u,v,w),1).
\end{equation*}
\end{proof}

\begin{rem}
By Proposition \ref{pp mu so(7)}, the representation $\gg\rightarrow\so(W,B_{\hh}|_W)$ of the quadratic Lie algebra $(\gg,B_{\hh}|_{\gg})$ together with the non-trivial cubic term on $W$ given by a multiple of the cross-product is of Lie type in the sense of Kostant \cite{Kostant99}.
\end{rem}
\vspace{0.1cm}

We now give the main result of this section.
\vspace{0.1cm}

\begin{thm} \label{thm O is special}
The representation $\rho : \hh \rightarrow \so(\mathbb{O},B)$ of the quadratic Lie algebra $(\hh,B_{\hh})$ is a special orthogonal representation.
\end{thm}
\begin{proof}
We want to show Equation \eqref{eq CS}. Let $u,v,w\in {\rm Im}(\mathbb{O})$. We have
\begin{align*}
    \mu_{\mathbb{O}}(u,v)(w)+\mu_{\mathbb{O}}(u,w)(v)&=\frac{8}{9}(\mu_{\rm Im}(u,v)(w)+\mu_{\rm Im}(u,w)(v))+\frac{1}{18}(c_{u\times v}(w)+c_{u\times w}(v))\\
    &=-\frac{2}{9}([w,[u,v]]+[v,[u,w]])+\frac{1}{9}((u\times v)\times w+(u\times w)\times v)\\
    &=-\frac{8}{9}(w\times (u\times v)+v\times (u\times w))+\frac{1}{9}((u\times v)\times w+(u\times w)\times v)\\
    &=w\times(v\times u)+v\times (w\times u)
\end{align*}
and by Equation \eqref{eq Malcev id}, we obtain
$$\mu_{\mathbb{O}}(u,v)(w)+\mu_{\mathbb{O}}(u,w)(v)=B(u,v)w+B(u,w)v-2B(v,w)u$$
and so \eqref{eq CS} is satisfied for $u,v,w\in {\rm Im}(\mathbb{O})$. We have
\begin{align*}
    \mu_{\mathbb{O}}(u,v)(1)+\mu_{\mathbb{O}}(u,1)(v)&=\frac{1}{18}c_{u\times v}(1)+\frac{1}{6}c_u(v)=B(u,v),\\
    \mu_{\mathbb{O}}(1,v)(w)+\mu_{\mathbb{O}}(1,w)(v)&=-\frac{1}{6}(c_v(w)+c_v(w))=-2B(v,w),
\end{align*}
and so \eqref{eq CS} is satisfied whenever two elements $u,v$ or $w$ are in ${\rm Im}(\mathbb{O})$. Finally, since
$$ 2\mu_{\mathbb{O}}(u,1)(1)=-2u,\qquad \mu_{\mathbb{O}}(1,u)(1)=u,$$   
then \eqref{eq CS} is satisfied whenever $u,v$ or $w$ is in ${\rm Im}(\mathbb{O})$ and so \eqref{eq CS} is satisfied for all $u,v,w\in \mathbb{O}$.
\end{proof}

By Theorems \ref{thm LSA from special orthogonal rep} and \ref{thm O is special} we have a Lie superalgebra $\tilde{\mathfrak{f}}$ of the form
$$\tilde{\mathfrak{f}}:=\hh\oplus \sl(2,k)\oplus \mathbb{O} \otimes k^2.$$
This is an exceptional simple Lie superalgebra of type $F_4$ in the Kac notation.
\vspace{0.2cm}

\begin{rem}
If $k=\mathbb{R}$, Serganova (see \cite{Serganova1983}) showed that there are four real forms of $F_4$. In particular, two of them have an even part isomorphic to $\so(7)\oplus \sl(2,\mathbb{R})$ (resp. $\so(4,3)\oplus \sl(2,\mathbb{R})$) and an odd part isomorphic to the tensor product of the spinor representation of $\so(7)$ (resp. $\so(4,3)$) and $\mathbb{R}^2$. In our construction, if $\mathbb{O}$ is the compact or the split octonion algebra, both real forms of $F_4$ are obtained by our construction.
\end{rem}

Since the representation $\hh \rightarrow \so(\mathbb{O},B)$ is special, we calculate its covariants $\psi_{\mathbb{O}}$, $Q_{\mathbb{O}}$ and the Mathews identities they satisfy. Since $\mathbb{O}$ is of dimension $8$, both-sides of Equations \eqref{id mat2} and \eqref{id mat3} vanish identically. However, both sides of the identities \eqref{id premat} and \eqref{id mat1} do no vanish identically. More precisely, up to constants, $\mu_{\mathbb{O}}\wedge_{\rho}\psi_{\mathbb{O}}$ and $Q_{\mathbb{O}}\wedge Id_{\mathbb{O}}\in {\rm Alt}_5(\mathbb{O},\mathbb{O})$ are the Hodge duals of the trilinear covariant $\psi_{\mathbb{O}} \in {\rm Alt}_3(\mathbb{O},\mathbb{O})$ and $\mu_{\mathbb{O}}\circ\psi_{\mathbb{O}}$ and $Q_{\mathbb{O}} \wedge \mu_{\mathbb{O}}\in {\rm Alt}_6(\mathbb{O},\hh)$ are the Hodge duals of the moment map $\mu_{\mathbb{O}} \in {\rm Alt}_2(\mathbb{O},\hh)$.
\vspace{0.2cm}

\begin{pp}
Let $\mu_{\mathbb{O}},\psi_{\mathbb{O}},Q_{\mathbb{O}}$ be the covariants of the special orthogonal representation $\rho : \hh \rightarrow \so(\mathbb{O},B)$. We have
\begin{enumerate}[label=\alph*)]
\item $\psi_{\mathbb{O}}(v_1,v_2,v_3)=-\frac{1}{2}(v_1,v_2,v_3)+\phi(v_1,v_2,v_3)~$ and $~\psi_{\mathbb{O}}(v_1,v_2,1)=-v_1\times v_2~$  for all $v_1,v_2,v_3\in {\rm Im}(\mathbb{O})$.
\item $Q_{\mathbb{O}}(v_1,v_2,v_3,v_4)=\frac{2}{3}Q_{{\rm Im}}(v_1,v_2,v_3,v_4)~$ and $~Q_{\mathbb{O}}(v_1,v_2,v_3,1)=-4\phi(v_1,v_2,v_3)~$ for all $v_1,v_2,v_3,v_4\in {\rm Im}(\mathbb{O})$.
\item \begin{align}
        \eta^{-1}(Q_{\mathbb{O}})&=\frac{4}{q(e_2)q(e_3)}e_{1234}-\frac{4}{q(e_2)q(e_3)q(e_5)}e_{1278}+\frac{4}{q(e_2)q(e_3)q(e_5)}e_{1368}-\frac{4}{q(e_2)q(e_3)q(e_5)}e_{1467}\nonumber\\
        &+\frac{4}{q(e_2)q(e_5)}e_{1256}+\frac{4}{q(e_3)q(e_5)}e_{1357}+\frac{4}{q(e_2)q(e_3)q(e_5)}e_{1458}+\frac{4}{q(e_2)q(e_3)q(e_5)}e_{2358}\nonumber\\
        &-\frac{4}{q(e_2)q(e_3)q(e_5)}e_{2367}-\frac{4}{q(e_2)q(e_3)q(e_5)}e_{2457}-\frac{4}{q(e_2)^2q(e_3)q(e_5)}e_{2468}+\frac{4}{q(e_2)q(e_3)q(e_5)}e_{3456}\nonumber\\
        &-\frac{4}{q(e_2)q(e_3)^2q(e_5)}e_{3478}-\frac{4}{q(e_2)q(e_3)q(e_5)^2}e_{5678}.\label{decomp Q so(7)}
        \end{align}
\end{enumerate}
\end{pp}
\begin{proof}
$a)$ By Propositions \ref{pp mu so(7)} and \ref{pp covariants ImO}, we have
\begin{align*}
    \psi_{\mathbb{O}}(v_1,v_2,v_3)&=\frac{8}{9}\psi_{\rm Im}(v_1,v_2,v_3)+\frac{1}{18}(c_{v_1\times v_2}(v_3)+c_{v_2\times v_3}(v_1)+c_{v_3\times v_1}(v_1))\\
    &=-\frac{2}{3}(v_1,v_2,v_3)+\frac{1}{9}((v_1\times v_2)\times v_3+(v_2\times v_3)\times v_1+(v_3\times v_1)\times v_2)+\phi(v_1,v_2,v_3).
\end{align*}
Using Equations \eqref{eq link product and cross product} and \eqref{eq Jacobi G2}, we have
\begin{align*}
    (v_1\times v_2)\times v_3+(v_2\times v_3)\times v_1+(v_3\times v_1)\times v_2=-\frac{1}{4}J(v_1,v_2,v_3)=\frac{3}{2}(v_1,v_2,v_3)
\end{align*}
and so
$$\psi_{\mathbb{O}}(v_1,v_2,v_3)=-\frac{1}{2}(v_1,v_2,v_3)+\phi(v_1,v_2,v_3).$$
We also have
$$\psi_{\mathbb{O}}(v_1,v_2,1)=\frac{1}{18}c_{v_1\times v_2}(1)+\frac{1}{6}c_{v_2}(v_1)-\frac{1}{6}c_{v_1}(v_2)=-\frac{1}{3}v_1\times v_2+\frac{1}{3}v_2\times v_1-\frac{1}{3}v_1\times v_2=-v_1\times v_2.$$
\noindent
$b)$ Follows from $a)$ and Propositions \ref{pp formules lien BQ Bpsi dans le cas CS} and \ref{pp covariants ImO}.
\vspace{0.2cm}

\noindent
$c)$ Using $b)$, the decomposition of $Q_{\mathbb{O}}$ follows from the decompositions \eqref{decomp Q G2} and \eqref{decomp phi G2}.
\end{proof}

\begin{rem}
\begin{enumerate}[label=\alph*)]
\item We have
    $$Q_{\mathbb{O}}\wedge Q_{\mathbb{O}}(e_1\wedge\ldots\wedge e_8)=-224q(e_2)^2q(e_3)^2q(e_5)^2.$$
    If $char(k)\neq 7$, then $Q_{\mathbb{O}}\wedge Q_{\mathbb{O}}$ defines an orientation on $\mathbb{O}$
\item In the decomposition \eqref{decomp Q so(7)}, there are fourteen $4$-vectors of the form $e_{i_1}\wedge e_{i_2}\wedge e_{i_3}\wedge e_{i_4}$. The fourteen quadruples of indices $\lbrace i_1,i_2,i_3,i_4\rbrace$ appearing are not arbitrary. There is a numbering of the eight points of the affine space $(\mathbb{Z}_2)^3$ such that each quadruple corresponds to one of the fourteen affine planes.
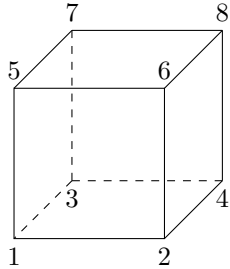
\begin{center}
\begin{tikzpicture}
\draw [dashed] (0,0,0)--(2,0,0);
\draw [dashed] (0,0,0)--(0,2,0);
\draw (2,0,0)--(2,2,0);
\draw (2,2,0)--(0,2,0);
\draw (0,0,2)--(2,0,2)--(2,2,2)--(0,2,2)--cycle; 
\draw [dashed] (0,0,0) -- (0,0,2); 
\draw (2,0,0) -- (2,0,2); 
\draw (2,2,0) -- (2,2,2); 
\draw (0,2,0) -- (0,2,2); 
\draw (0,0,0) node[below]{$3$};
\draw (0,0,2) node[below]{$1$};
\draw (0,2,0) node[above]{$7$};
\draw (0,2,2) node[above]{$5$};
\draw (2,0,0) node[below]{$4$};
\draw (2,0,2) node[below]{$2$};
\draw (2,2,0) node[above]{$8$};
\draw (2,2,2) node[above]{$6$};
\end{tikzpicture}
\end{center}
\end{enumerate}
\end{rem}

Suppose that $char(k)=0$ or $char(k)>7$. Define a quadratic form $B_{{\rm Alt}(\mathbb{O},\mathbb{O})}$ (resp. $B_{{\rm Alt}(\mathbb{O},\hh)}$) on ${\rm Alt}_i(\mathbb{O},\mathbb{O})\cong\Lambda^i(\mathbb{O})^*\otimes \mathbb{O}$ (resp. ${\rm Alt}_i(\mathbb{O},\hh)\cong\Lambda^i(\mathbb{O})^*\otimes \hh$) to be the tensor product of $B_{\Lambda^*}$ and $B$ (resp. $B_{\hh}$). For $f\in {\rm Alt}_i(\mathbb{O},\mathbb{O})$ define its Hodge dual $*f\in {\rm Alt}_{8-i}(\mathbb{O},\mathbb{O})$ to be the unique element which satisfies
$$\alpha\wedge_{B} *f=B_{{\rm Alt}(\mathbb{O},\mathbb{O})}(\alpha,f)Q_{\mathbb{O}}\wedge Q_{\mathbb{O}} \qquad \forall \alpha \in {\rm Alt}_i(\mathbb{O},\mathbb{O})$$
and for $f\in {\rm Alt}_i(\mathbb{O},\hh)$ define its Hodge dual $*f\in {\rm Alt}_{8-i}(\mathbb{O},\hh)$ to be the unique element which satisfies
$$\alpha\wedge_{B_{\hh}} *f=B_{{\rm Alt}(\mathbb{O},\hh)}(\alpha,f)Q_{\mathbb{O}}\wedge Q_{\mathbb{O}} \qquad \forall \alpha \in {\rm Alt}_i(\mathbb{O},\hh).$$

\begin{pp}
We have
\begin{enumerate}[label=\alph*)]
\item $*\psi_{\mathbb{O}}=-56 Q_{\mathbb{O}}\wedge Id=\frac{112}{3} \mu_{\mathbb{O}}\wedge_{\rho}\psi_{\mathbb{O}},$
\item $*\mu_{\mathbb{O}}=-56 Q_{\mathbb{O}}\wedge \mu_{\mathbb{O}}=-\frac{56}{3} \mu_{\mathbb{O}}\circ\psi_{\mathbb{O}}.$
\end{enumerate}
\end{pp}

\begin{proof}
$a)$ Let $\alpha\in{\rm Alt}_3(\mathbb{O},\mathbb{O})$. We have
\begin{equation}\label{eq proof hodge so(7) 1}
    B_{{\rm Alt}(\mathbb{O},\mathbb{O})}(\alpha,\psi_{\mathbb{O}})Q_{\mathbb{O}}\wedge Q_{\mathbb{O}}(e_1\wedge \ldots \wedge e_8)=-224\sum\limits_{i<j<k}\frac{1}{q(e_i)q(e_j)q(e_k)}B(\alpha(e_i,e_j,e_k),\psi_{\mathbb{O}}(e_i,e_j,e_k))q(e_2)^2q(e_3)^2q(e_5)^2.
\end{equation}
On the other hand
\begin{align*}
    \alpha\wedge_B Q_{\mathbb{O}}\wedge Id(e_1\wedge \ldots \wedge e_8)&=\frac{1}{144}\sum\limits_{\sigma \in S_8}sgn(\sigma)B(\alpha(e_{\sigma(1)},e_{\sigma(2)},e_{\sigma(3)}),e_{\sigma(4)})Q_{\mathbb{O}}(e_{\sigma(5)},e_{\sigma(6)},e_{\sigma(7)},e_{\sigma(8)}).
\end{align*}
Since $\alpha$ and $Q_{\mathbb{O}}$ are alternating and using the decomposition of Equation \eqref{decomp Q so(7)}, we have
\begin{equation}\label{eq proof hodge so(7) 2}
    \alpha\wedge_B Q_{\mathbb{O}}\wedge Id(e_1\wedge \ldots \wedge e_8)=\sum\limits_{\sigma \in S}sgn(\sigma)B(\alpha(e_{\sigma(1)},e_{\sigma(2)},e_{\sigma(3)}),e_{\sigma(4)})Q_{\mathbb{O}}(e_{\sigma(5)},e_{\sigma(6)},e_{\sigma(7)},e_{\sigma(8)})
\end{equation}
where
$$S=\lbrace\sigma\in S_8 ~ | ~ \sigma(1)<\sigma(2)<\sigma(3),\quad (\sigma(5),\sigma(6),\sigma(7),\sigma(8))\in \lbrace (1,2,3,4),\ldots,(5,6,7,8) \rbrace \rbrace.$$
We have $|S|=56$ and each summand in \eqref{eq proof hodge so(7) 1} correspond to one summand in \eqref{eq proof hodge so(7) 2}. A straightforward calculation gives
\begin{align*}
    \alpha\wedge_B Q_{\mathbb{O}}\wedge Id(e_1\wedge \ldots \wedge e_8)&=4\sum\limits_{i<j<k}\frac{1}{q(e_i)q(e_j)q(e_k)}B(\alpha(e_i,e_j,e_k),\psi_{\mathbb{O}}(e_i,e_j,e_k))q(e_2)^2q(e_3)^2q(e_5)^2
\end{align*}
and so
$$\alpha\wedge_B Q_{\mathbb{O}}\wedge Id=-\frac{1}{56}B_{{\rm Alt}(\mathbb{O},\mathbb{O})}(\alpha,\psi_{\mathbb{O}})Q_{\mathbb{O}}\wedge Q_{\mathbb{O}}.$$

\noindent
$b)$ Let $\alpha\in{\rm Alt}_2(\mathbb{O},\hh)$. We have
$$B_{{\rm Alt}(\mathbb{O},\hh)}(\alpha,\mu_{\mathbb{O}})Q_{\mathbb{O}}\wedge Q_{\mathbb{O}}(e_1\wedge\ldots\wedge e_8)=-224\sum\limits_{i<j}\frac{1}{q(e_i)q(e_j)}B_{\hh}(\alpha(e_i,e_j),\mu_{\mathbb{O}}(e_i,e_j))q(e_2)^2q(e_3)^2q(e_5)^2.$$
On the other hand
\begin{align*}
    \alpha\wedge_{B_{\hh}} Q_{\mathbb{O}}\wedge \mu_{\mathbb{O}}(e_1\wedge \ldots\wedge e_8)&=\frac{1}{96}\sum\limits_{\sigma\in S_8}sgn(\sigma)B_{\hh}(\alpha(e_{\sigma(1)},e_{\sigma(2)}),\mu_{\mathbb{O}}(e_{\sigma(3)},e_{\sigma(4)}))Q_{\mathbb{O}}(e_{\sigma(5)},e_{\sigma(6)},e_{\sigma(7)},e_{\sigma(8)}).
\end{align*}
Since $\alpha, \mu_{\mathbb{O}}$ and $Q_{\mathbb{O}}$ are alternating and using the decomposition of Equation \eqref{decomp Q so(7)}, we have
$$\alpha\wedge_{B_{\hh}} Q_{\mathbb{O}}\wedge \mu_{\mathbb{O}}(e_1\wedge \ldots\wedge e_8)=\sum\limits_{\sigma\in S}sgn(\sigma)B_{\hh}(\alpha(e_{\sigma(1)},e_{\sigma(2)}),\mu_{\mathbb{O}}(e_{\sigma(3)},e_{\sigma(4)}))Q_{\mathbb{O}}(e_{\sigma(5)},e_{\sigma(6)},e_{\sigma(7)},e_{\sigma(8)})$$
where
$$S=\lbrace\sigma\in S_8 ~ | ~ \sigma(1)<\sigma(2),\quad \sigma(3)<\sigma(4),\quad (\sigma(5),\sigma(6),\sigma(7),\sigma(8))\in \lbrace (1,2,3,4),\ldots,(5,6,7,8) \rbrace \rbrace.$$
We have $|S|=84$ and, using Equation \eqref{eq moment map O of cyclic sum}, a straightforward calculation gives
$$\alpha\wedge_{B_{\hh}} Q_{\mathbb{O}}\wedge \mu_{\mathbb{O}}(e_1\wedge \ldots\wedge e_8)=4\sum\limits_{i<j}\frac{1}{q(e_i)q(e_j)}B_{\hh}(\alpha(e_i,e_j),\mu_{\mathbb{O}}(e_i,e_j))q(e_2)^2q(e_3)^2q(e_5)^2,$$
and so
$$\alpha\wedge_{B_{\hh}} Q_{\mathbb{O}}\wedge \mu_{\mathbb{O}}(e_1\wedge \ldots\wedge e_8)=-\frac{1}{56}B_{{\rm Alt}(\mathbb{O},\hh)}(\alpha,\mu_{\mathbb{O}})Q_{\mathbb{O}}\wedge Q_{\mathbb{O}}.$$
\end{proof}

\begin{rem}
One can show similarly that, up to a constant,
the identity is the Hodge dual of $Q_{\mathbb{O}}\wedge \psi_{\mathbb{O}}$.
\end{rem}

\footnotesize
\bibliographystyle{alpha}
\bibliography{ExceptionalLSAs}
\end{document}